\documentclass{smfart}

\RequirePackage[T1]{fontenc}
\RequirePackage{amsfonts,latexsym,amssymb}

\RequirePackage{smfenum}
\RequirePackage[frenchb]{babel}
\addto\extrasfrenchb{\bbl@nonfrenchitemize
\bbl@nonfrenchspacing}

\RequirePackage{mathrsfs}
\let\mathcal\mathscr
\usepackage[matrix,arrow]{xy}
\usepackage{url}
\usepackage{accents}
\usepackage{enumitem}
\usepackage{fge}
\usepackage{bbm}

\theoremstyle{plain}
\newtheorem{prop}[equation]{\propname}
\newtheorem{theo}[equation]{\theoname}

\newtheorem{lemm}[equation]{\lemmname}
\theoremstyle{definition}
\theoremstyle{remark}

\newtheorem{rema}[equation]{\remaname}

\let\cal\mathcal
\let\goth\mathfrak

\def\Qbar{\overline{\bf Q}}
\def\Q{{\bf Q}} \def\Z{{\bf Z}}
\def\C{{\bf C}}
\def\B{{\bf B}}

\def\R{{\bf R}}
\def\O{{\cal O}}
\def\dual{{\boldsymbol *}}

\begin{document}
\title{La formule limite de Kronecker}
\author{Pierre Colmez}
\address{C.N.R.S., IMJ-PRG, Sorbonne Universit\'e, 4 place Jussieu,
75005 Paris, France}
\email{pierre.colmez@imj-prg.fr}
\begin{abstract}
Nous faisons quelques remarques sur les liens entre la formule limite de Kronecker
pour les corps totalement r\'eels, les p\'eriodes (normalis\'ees) des vari\'et\'es ab\'eliennes CM,
et les unit\'es des extensions ab\'eliennes des corps CM.
\end{abstract}
\begin{altabstract}
We make some remarks about the relation between the Kronecker limit formula for totally real fields, 
(normalized) periods of CM abelian varieties, and units in abelian extensions of CM fields.
\end{altabstract}

\maketitle

Cette petite note est une s\'erie de commentaires sur~\cite{fdp},
en particulier sur
 le quatri\`eme point de~\cite[\S\,0.7]{fdp}, 
qui explicite ce que la conjecture~\cite[conj.\,04]{fdp}
(ou \cite[conj.\,II.2.11]{fdp})
implique pour la moyenne, sur tous les types CM possibles,
des hauteurs de Faltings des vari\'et\'es ab\'eliennes CM
par l'anneau des entiers $\O_K$ d'un corps CM $K$ donn\'e: si $\chi$ est le caract\`ere
quadratique associ\'e \`a $K/F$ o\`u $F$ est le sous-corps r\'eel maximal de $K$, 
on a\footnote{Les parenth\`eses de
la formule de~\cite{fdp} sont mal plac\'ees.}
\begin{equation}\label{truc0}
\frac{1}{2^{[F:{\mathbb Q}]}}\sum_{\Phi\in\Phi(K)}
h_{\rm Fal}(X_\Phi)=-\frac{1}{2}\Big(\frac{L'(\chi,0)}{L(\chi,0)}+\frac{1}{2}\log{\bf f}_\chi\Big)
-\frac{[F:{\mathbb Q}]}{2}\log 2\pi
\end{equation}
Cette formule est maintenant un th\'eor\`eme (\cite{Y} pour $[K:\Q]=4$ et \cite{AGHM,YZ}
pour le cas g\'en\'eral). Dans~\cite[\S\,0.7]{fdp},
cette formule est suivie du commentaire suivant:
\begin{quote}
[...] le second membre s'exprime \`a l'aide de la formule limite de Kronecker pour $F$. La conjecture 0.4
laisse donc penser que la fonction apparaissant dans la formule limite de Kronecker pour un corps
totalement r\'eel a des propri\'et\'es arithm\'etiques, au moins aux points CM, tout-\`a-fait
similaires \`a celles de la classique fonction $\log|\Delta|$ [...]
\end{quote}
Nous explicitons ci-dessous (formule~(\ref{truc})) ce qu'il en est exactement\footnote{Ce que j'avais
en t\^ete \`a l'\'epoque \'etait beaucoup trop optimiste.}. Nous en profitons
aussi pour donner une d\'efinition g\'eom\'etrique (th.\,\ref{per1} et prop.\,\ref{CM5})
de p\'eriodes ``absolues'' des vari\'et\'es
ab\'eliennes CM; notons que Yoshida~\cite{yoshida1} a formul\'e une conjecture
donnant une construction analytique de telles p\'eriodes.
Enfin, dans l'appendice, nous montrons (prop.\,\ref{per3}) que le
(ii) de la conjecture~\cite[conj.\,II.2.11]{fdp} (concernant les hauteurs de Faltings)
implique ``la moiti\'e'' de la conjecture compl\`ete (la conjecture
compl\`ete est connue dans le cas ab\'elien gr\^ace \`a~\cite[th.\,III.3.9]{fdp} compl\'et\'e
par~\cite{O}).

\section{Fonctions z\^eta de Dedekind en $s=0$}
On note $\Qbar$ la cl\^oture alg\'ebrique de $\Q$ dans $\C$.
Si $L\subset\Qbar$ est un corps de nombres, on note $\O_L$ l'anneau de ses entiers, 
$U_L$ le groupe des unit\'es de $\O_L$, ${\rm Cl}_L$ le groupe
des classes d'id\'eaux fractionnaires de $L$, $h_L$ son cardinal, $w_L$ le nombre de racines
de l'unit\'es dans $L$. On note $H_L$ l'ensemble ${\rm Hom}(L,\Qbar)$
des plongements de $L$ dans $\Qbar$; comme on a suppos\'e $L\subset\Qbar$,
cela fait que $H_L$ contient des \'el\'ements privil\'egi\'es ${\rm id}$ et la
conjugaison complexe $c$ (qui peut co\"{\i}ncider avec ${\rm id}$).

Soient $r_1$ le nombre de places r\'eelles de $L$ et $r_2$ le nombre
de ses places complexes, et donc $[L:\Q]=r_1+2r_2$.
Si $v$ est une place archim\'edienne de $L$, choisissons $\sigma_v\in H_L$ induisant $v$,
et posons $e_v=1$ si $v$ est r\'eelle et $e_v=2$ si $v$ est complexe.  L'injection
de $L$ dans $\Qbar$ induit une place privil\'egi\'ee $v_0$.
Si $u_1,\dots,u_{r_1+r_2-1}$ est une famille libre d'\'el\'ements de $U_L$, le sous-groupe
$\langle u_1,\dots,u_{r_1+r_2-1}\rangle$ qu'ils engendrent est d'indice fini dans $U_L$.
On d\'efinit le r\'egulateur $R_L$ de $U_L$ par la formule
$$
R_L:=\frac{1}{[U_L:\langle u_1,\dots,u_{r_1+r_2-1}\rangle]}
\big|\det(\log|\sigma_v(u_i)|^{e_v})_{v\neq v_0,\,1\leq i\leq r_1+r_2-1}\big|
$$
(Comme on a divis\'e par l'indice dans tout le groupe des unit\'es et pas dans sa partie libre,
notre $R_L$ correspond au $\frac{R_L}{w_L}$ habituel.)

On note $\zeta_L$ la fonction z\^eta de Dedekind de $L$; c'est une fonction analytique sur $\C$, 
holomorphe en dehors d'un p\^ole simple en $s=1$; son comportement en $s=0$ est donn\'e
par la formule analytique du nombre de classes
$$\zeta_L(s)\sim -{h_LR_L}s^{r_1+r_2-1}$$

Soient $K$ un corps CM et $F$ son sous-corps totalement r\'eel maximal.  
On note $d$ le degr\'e de $F$ sur $\Q$.
Alors $U_F$ est d'indice fini dans $U_K$,
et on a 
$$[U_K:U_F]=2^{d-1}\frac{R_F}{R_K}$$

Au voisinage de $0$, on a 
$$\zeta_F(s)=-h_F\big({R_F}s^{d-1}+\gamma_F s^d\big)+O(s^{d+1}),
\quad{\text{avec $\gamma_F\in\R$.}}$$
\begin{rema}\label{Kr1}
On a $\gamma_\Q=\log\sqrt{ 2\pi}$, mais la nature arithm\'etique de
$\gamma_F$ n'est pas claire si $F\neq\Q$.
\end{rema}
Soit $\chi$ le caract\`ere
de Dirichlet de $F$ correspondant \`a l'extension quadratique $K/F$.
La formule (\ref{truc0}) est \'equivalente (via le (ii) de \cite[th.\,0.3]{fdp})
\`a l'identit\'e\footnote{${\rm ht}$ est la forme lin\'eaire fabriqu\'ee
\`a partir des logarithmes de p\'eriodes de vari\'et\'es ab\'eliennes CM, qui est d\'efinie
dans~\cite[th.\,0.3]{fdp} (cf.~aussi le dernier point de la prop.\,\ref{CM5} et la
note~\ref{per20}.)}
$$L(\chi,s)=L(\chi,0)(1+{\rm ht}(\chi)s)+O(s^2)$$
Comme $\zeta_K(s)=\zeta_F(s)L(\chi,s)$, on en tire
$$\zeta_K(s)=-h_FL(\chi,0)\big({R_F}s^{d-1}+(\gamma_F
+{R_F}{\rm ht}(\chi))s^d\big)+O(s^{d+1})$$
En comparant avec la formule analytique du nombre de classes pour $K$, cela
donne 
\begin{align}
L(\chi,0)&=\frac{h_KR_K}{h_FR_F}=\frac{h_K}{h_F}\ \frac{2^{d-1}}{[U_K:U_F]}\notag\\
\label{truc2}
\zeta_K(s)&=-{h_KR_K}s^{d-1}-\Big(\frac{2^{d-1}h_K}{[U_K:U_F]}\gamma_F
+h_K{R_K}{\rm ht}(\chi)\Big)s^d+O(s^{d+1})
\end{align}

\section{Fonctions z\^eta d'Epstein en $s=0$}
Si $\Lambda$ est un $\O_F$-r\'eseau de $\C\otimes_\Q F$, on pose
$$E(\Lambda,s)=\sum\nolimits'_{\omega\in\Lambda/U_F}\frac{1}{|N(\omega)|^{2s}}$$
 o\`u $N:\C\otimes F\to\C$
est induit par ${\rm N}_{F/\Q}:F\to\Q$.  Alors, $E(\Lambda,s)$ admet
un prolongement analytique \`a tout le plan complexe, holomorphe en dehors d'un p\^ole simple en $s=1$,
et le comportement au voisinage de $0$ est donn\'e par la formule limite de Kronecker:
$$E(\Lambda,s)=-{R_F}s^{d-1}-(\Psi_F(\Lambda)+2^{d-1}\gamma_F)s^d+O(s^{d+1})$$
o\`u $\Psi_F$ est une fonction sur les $\O_F$-r\'eseaux v\'erifiant
$$\Psi_F(\alpha\Lambda)=\Psi_F(\Lambda)+{R_F}\log|{\rm N}(\alpha)|^2,
\quad{\text{si $\alpha\in (\C\otimes_\Q F)^\dual$}}$$ 
et la constante
$2^{d-1}\gamma_F$ est l\`a pour rendre les formules suivantes plus esth\'etiques.
\begin{rema}\label{Kr2}
(i) On peut trouver $\omega_1,\omega_2$ tels que $\Lambda={\goth a}\omega_1\oplus\O_F\omega_2$,
o\`u ${\goth a}$ est un id\'eal de $\O_F$ dont l'image dans ${\rm Cl}_K$ ne d\'epend que de $\Lambda$.
Si $z=\frac{\omega_2}{\omega_1}$, alors 
$$\Psi_F(\Lambda)=\Psi_F({\goth a}\oplus\O_F z)+ R_F\log|{\rm N}(\omega_1)|^2$$
  Les $z\mapsto \Psi_F({\goth a}\oplus\O_F z)$, 
pour $\dot{\goth a}\in{\rm Cl}_F$ sont les branches de la forme automorphe 
analytique r\'eelle dont la transform\'ee de Mellin est $\zeta_F(s)\zeta_F(s+1)$.

(ii) Si $F=\Q$, on peut \'ecrire $\Lambda=\omega_1(\Z+\Z z)$ avec ${\rm Im}(z)>0$,
et on a $$\Psi_\Q(\Z+\Z z)=\log(\sqrt{2\pi}\,|\eta(z)|^2)$$
o\`u $\eta(z)=q^{1/24}\prod_{n\geq 1}(1-q^n)$ est la forme modulaire de poids $\frac{1}{2}$ habituelle.
\end{rema}

Si $\dot{\goth a}\in{\rm Cl}_K$, fixons un id\'eal ${\goth a}$ de $\O_K$ dont
l'image dans ${\rm Cl}_K$ est $\dot{\goth a}$.
Alors 
$$\zeta_{\dot{\goth a}}(s)=\sum_{{\goth b}\sim{\goth a}}\frac{1}{N{\goth b}^s}=
\frac{1}{[U_K:U_F]}\frac{1}{N{\goth a}^s}\sum_{\alpha\in{\goth a}^{-1}/U_F}
\frac{1}{|N(\alpha)|^{2s}}
=\frac{1}{[U_K:U_F]}\frac{1}{N{\goth a}^s} E({\goth a}^{-1},s)$$
Au voisinage de $0$, on a donc
\begin{equation}\label{truc3}
\zeta_{\dot{\goth a}}(s)=-{R_K}s^{d-1}-
\tfrac{1}{[U_K:U_F]}(\Psi_F(\dot{\goth a})+2^{d-1} \gamma_F)s^d+O(s^{d+1})
\end{equation}
o\`u $\Psi_F(\dot{\goth a})
:=\Psi_F({\goth a}^{-1})+{R_F}\log{\rm N}{\goth a}$.
En comparant les formules (\ref{truc2}) et (\ref{truc3}) pour 
$\zeta_K(s)=\sum_{\dot{\goth a}\in{\rm Cl}_K}\zeta_{\dot{\goth a}}(s)$,
on en tire l'identit\'e
\begin{equation}\label{truc7}
\frac{1}{[U_K:U_F]}
\sum_{\dot{\goth a}\in{\rm Cl}_K}\Psi_F(\dot{\goth a})= h_K{R_K}{\rm ht}(\chi)
\end{equation}
Maintenant, si $\eta:{\rm Cl}_K\to \C^\dual$ est un caract\`ere,
le comportement de $L(\eta,s)$ au voisinage de $0$ est donn\'e par la conjecture de Stark:
si $H$ est le corps de classes de Hilbert de $K$, et 
si $U_H^\eta$ est le sous-espace de $U_H\otimes\Q(\eta)$ sur lequel ${\rm Gal}(H/K)={\rm Cl}_K$
agit par $\eta$, alors
il existe $A(\eta)\in\Q(\eta)$ tel que
$$L(\eta,s)=A(\eta)R(U_H^\eta) s^d+O(s^{d+1})$$
o\`u $R(U_H^\eta)$ est le r\'egulateur de $U_H^\eta$ (un d\'eterminant $d\times d$ de combinaisons
lin\'eaires de logarithmes d'unit\'es de $\O_H$).
Comme $L(\eta,s)=\sum_{\dot{\goth a}}\eta(\dot{\goth a})\zeta_{\dot{\goth a}}(s)$, 
une application de la formule d'inversion
de Fourier sur ${\rm Cl}_K$ fournit l'identit\'e (conjecturale)
\begin{equation}\label{truc}
\Psi_F(\dot{\goth a})=
[U_K:U_F]{R_K}{\rm ht}(\chi)+\frac{1}{h_K}\sum_{\eta\neq 1}\eta(\dot{\goth a})^{-1}
A(\eta)R(U_H^\eta)
\end{equation}

\begin{rema}\label{CM1}
(i) Si $F=\Q$, l'identit\'e~(\ref{truc})
 est un th\'eor\`eme gr\^ace au lien entre $\eta$, les p\'eriodes
des courbes elliptiques CM et les unit\'es elliptiques.

(ii) Si $F\neq\Q$, l'identit\'e~(\ref{truc})
montre que $\Psi_F$ prend des valeurs int\'eressantes aux points
sp\'eciaux, mais que la situation est nettement plus compliqu\'ee que pour $\Q$
car ces valeurs font intervenir des d\'eterminants $d\times d$ de logarithmes
de nombres significatifs, ce qui ne donne pas un acc\`es direct \`a ces nombres.

(iii) En dehors du cas $F=\Q$, on peut prouver l'identit\'e (\ref{truc}) si
$|{\rm Cl}_K|=1,2$.  Si $|{\rm Cl}_K|=1$, cela se r\'eduit \`a
la formule~(\ref{truc7}),
et si
$|{\rm Cl}_K|=2$, on obtient
$$\Psi_F(\dot{\goth a})=[U_K:U_F]{R_K}{\rm ht}(\chi)\pm \frac{1}{2}\frac{h_HR_H}{h_KR_K}
=[U_K:U_F]{R_K}{\rm ht}(\chi)\pm \frac{1}{2}{h^-_HR^-_H}$$
o\`u $\pm=+$ si $\dot{\goth a}$ est la classe des id\'eaux principaux
et $\pm=-$ sinon.

(iv) Le membre de droite de l'identit\'e (\ref{truc}) n'a pas l'air homog\`ene
car le terme $R_K{\rm ht}(\chi)$ est un produit d'un d\'eterminant $(d-1)\times(d-1)$
de logarithmes de conjugu\'es d'unit\'es par le logarithme d'une p\'eriode, alors que tous les autres
termes sont des d\'eterminants $d\times d$ de logarithmes de conjugu\'es d'unit\'es.
Nous expliquons ci-dessous (prop.\,\ref{Kr6}) comment \'ecrire naturellement
$R_K{\rm ht}(\chi)$ comme un d\'eterminant $d\times d$.
\end{rema}

\def\CMm{{\cal C}\hskip-.9mm{\cal M}_{\hskip-.4mm {\text{-\hskip-.2mm-}}}}
\def\CM{{\cal C}\hskip-.9mm{\cal M}}

\section{Corps CM}
Notons $\Q^{\rm CM}\subset\Qbar$ le compos\'e de tous les corps CM
et, si $K$ est un corps CM, notons $G_K^{\rm CM}$ le groupe ${\rm Gal}(\Q^{\rm CM}/K)$.

Soit $\CM$ l'espace des $a:G_\Q\to\Z$, localement constantes,
 se factorisant par $G_\Q^{\rm CM}$,
telles que $g\mapsto a(g)+a(cg)$ soit constante,
o\`u $c$ est la conjugaison complexe
(qui est dans le centre de $G_\Q^{\rm CM}$).
Alors $\CM$ contient, comme sous-module d'indice~$2$,
la somme directe de l'espace des fonctions constantes
et de $\CMm$, espace des $a$ telles que $a(cg)=-a(g)$ pour tout $g\in G_\Q$.
Soient $\CM^0\subset\CM$ et
 $\CMm^0\subset\CMm$ les sous-espaces des fonctions centrales.

Alors $\C\otimes \CM^0$ admet pour base les caract\`eres d'Artin dont la fonction~$L$
ne s'annule pas en $0$.
Soient $\mu_{\rm Art}$ et, si $s\in\C$, $Z(\cdot,s)$ les formes $\C$-lin\'eaires
sur $\CM^0$ d\'efinies par $\mu_{\rm Art}(\chi)=\log{\goth f}_\chi$ (o\`u ${\goth f}_\chi$
est le conducteur de $\chi$) et $Z(\chi,s)=\frac{L'(\chi,s)}{L(\chi,s)}$.
L'inclusion $\Qbar\hookrightarrow\Qbar_p$ induit une inclusion
$G_{\Q_p}\hookrightarrow G_\Q$, et les formes lin\'eaires ci-dessus se d\'ecomposent
sous la forme
$$\mu_{\rm Art}=\sum_{p\in{\cal P}}\mu_{{\rm Art},p}\log p,\quad
Z(\cdot,s)=-\sum_{p\in{\cal P}} Z_p(\cdot,s)\log p$$
o\`u $\mu_{{\rm Art},p}$ et $Z_p(\cdot,s)$ sont des applications lin\'eaires
sur ${\rm LC}(G_{\Q_p},\Z)$, \`a valeurs dans $\Q$ et $\Q(p^{-s})$ respectivement.

On suppose que $F\subset K\subset\Qbar$, i.e.~que l'on dispose d'un plongement privil\'egi\'e
de $F$ et $K$ dans $\Qbar$.  Si $\sigma,\tau\in H_K$,
soit $b_{K,\sigma,\tau}\in\CMm$ la fonction sur $G_\Q$ donn\'ee par 
la formule\footnote{Elle est reli\'ee aux $a_{K,\sigma,\tau}$ de~\cite{fdp} par
la formule
$b_{K,\sigma,\tau}=a_{K,\sigma,\tau}-a_{K,\sigma, c\tau}$.}
\begin{equation}\label{truc5}
b_{K,\sigma,\tau}(g)=\begin{cases} 1&{\text{si $g\sigma=\tau$,}}\\
-1&{\text{si $g\sigma=c\tau$,}}\\
0&{\text{sinon.}}
\end{cases}
\end{equation}
On a $b_{K,\sigma,\tau}=-b_{K,\sigma,c\tau}=-b_{K,c\sigma,\tau}=b_{K,c\sigma,c\tau}$.

Si $\Phi\subset H_K$ est un type CM, soit 
$$b_{K,\sigma,\Phi}=\sum_{\tau\in\Phi} b_{K,\sigma,\tau}$$

\section{Anneaux de p\'eriodes}
On fixe un plongement
$\Qbar\hookrightarrow \Qbar_p$ pour chaque $p$.
Notons $\B_{{\rm dR},p}$ l'anneau des p\'eriodes $p$-adiques.
Soit ${\cal V}:=\{\infty, p\in{\cal P}\}$ l'ensemble des places de $\Q$.
Si $v\in{\cal V}$, on d\'efinit $\B_v$ par $\B_\infty=\C$ et
$\B_p=\B_{{\rm dR},p}$, si $p\in{\cal P}$. Finalement, on pose:
$$\B:=\prod\nolimits_v\B_v$$

Si ${\rm LC}$ d\'esigne
l'espace des fonctions localement constantes, on a des isomorphismes
$$\big(\B\otimes\Qbar\big)^\dual=
\varinjlim_{[L:\Q]<\infty}\big(\B\otimes L\big)^\dual\cong
{\rm LC}(G_\Q,\B^\dual)$$
On voit $\Qbar^\dual$ comme un sous-groupe de $\big(\B\otimes\Qbar\big)^\dual$
en envoyant $\alpha$ sur $1\otimes\alpha$; 
via l'isomorphisme ci-dessus,
cela correspond \`a envoyer $\alpha$ sur la fonction localement constante
$g\mapsto \alpha^g$.  L'injection $\Qbar^\dual \to {\rm LC}(G_\Q,\B^\dual)$
est $G_\Q$ \'equivariante, si on fait agir $h\in G_\Q$ sur ${\rm LC}(G_\Q,\B^\dual)$
par $$(h*\phi)(g)=\phi(gh)$$

On note $U_{\Qbar}\subset \Qbar^\dual$ la limite inductive des $U_L$.
On note $U^{\rm CM}$ la limite inductive des 
$U_F^+$, pour $F$ totalement r\'eel.
On d\'efinit $\iota:\Qbar^\dual\mapsto (\B\otimes\Qbar)^\dual$ en envoyant $\lambda$
sur la fonction
$g\mapsto \lambda^g\lambda^{cg}$, et donc $\iota({\Qbar}^\dual)$ n'est pas un sous-groupe
de $\Qbar^\dual\hookrightarrow (\B\otimes\Qbar)^\dual$. On note $\iota({\Qbar}^\dual)_0$
le sous-groupe des $\iota(\lambda)$ v\'erifiant $v_p(\lambda^g\lambda^{cg})=0$
pour tout $g\in G_\Q$; en particulier, $\iota(U_{\Qbar})\subset\iota(\Qbar^\dual)_0$.

La restriction de $\iota$ \`a $\Q^{\rm CM}$ est $\lambda\mapsto\lambda\lambda^c$
car $\lambda^{cg}=\lambda^{gc}$ pour tout $g$, si $\lambda\in U^{\rm CM}$.
Il s'ensuit que $\iota((\Q^{\rm CM})^\dual)\subset (\Q^{\rm CM})^\dual$
et que $\iota((\Q^{\rm CM})^\dual)_0\subset U^{\rm CM}\subset
\Qbar^\dual\hookrightarrow (\B\otimes\Qbar)^\dual$.

On dispose d'une application norme normalis\'ee
\begin{equation}\label{ncm}
{\rm N}^{\rm CM}:(\B\otimes\Qbar)^\dual\to(\B\otimes\Q^{\rm CM})^\dual\otimes_\Z\Q
\end{equation}
d\'efinie, via les identifications
$(\B\otimes\Qbar)^\dual\cong{\rm LC}(G_\Q,\B^\dual)$
et $(\B\otimes\Q^{\rm CM})^\dual\cong{\rm LC}(G^{\rm CM}_\Q,\B^\dual)$
par la formule suivante, pour $\phi$ fixe par $G_L$ avec $L/\Q$ galoisienne finie,
et $K$ plus grand sous-corps CM de $L$:
$$({\rm N}^{\rm CM}\phi)(g)=\Big(\prod_{h\in G_K/G_L}\phi(gh)\Big)^{1/[L:K]}$$
(l'exposant $1/[L:K]$ est la raison pour laquelle on a tensoris\'e
$(\B\otimes\Q^{\rm CM})^\dual$ par $\Q$; on v\'erifie facilement que le r\'esultat est ind\'ependant
de $L$ gr\^ace \`a cet exposant $1/[L:K]$).
Si $\lambda\in L^\dual$,
on a, car $K$ est un corps CM,
 $$\prod_{h\in G_K/G_L}\lambda^{gh}\lambda^{cgh}=
({\rm N}_{L/K}\lambda)^g({\rm N}_{L/K}\lambda)^{cg}
=(({\rm N}_{L/K}\lambda)({\rm N}_{L/K}\lambda)^c)^g$$
Il s'ensuit que ${\rm N}^{\rm CM}(\iota({\Qbar}^\dual)_0)\subset U^{\rm CM}\otimes \Q$,
ce qui induit un morphisme
$${\rm N}^{\rm CM}:(\B\otimes\Qbar)^\dual/\iota({\Qbar}^\dual)_0
\to\big((\B\otimes\Q^{\rm CM})^\dual/U^{\rm CM}\big)\otimes_\Z\Q$$
induisant l'identit\'e sur
$(\B\otimes\Q^{\rm CM})^\dual/U^{\rm CM}\subset
(\B\otimes\Qbar)^\dual/\iota({\Qbar}^\dual)_0$.

\section{P\'eriodes des vari\'et\'es ab\'eliennes CM}


\medskip
Si $\Phi\in H_K$ est un type CM, choisissons une vari\'et\'e ab\'elienne
$X_\Phi$, d\'efinie sur une extension finie $L$ de $K$, \`a multiplication complexe par $\O_K$,
de type CM $\Phi$, ayant bonne r\'eduction partout (c'est possible, quitte \`a agrandir $L$).
On fixe un mod\`ele ${\cal X}_\Phi$ de $X_\Phi$ sur $\O_K$.  Si $\sigma\in H_K$,
soit $H^\sigma({\cal X}_\Phi)$ le
$\O_L$-module des $\omega\in H^1_{\rm dR}({\cal X}_\Phi)$ sur lesquels $\alpha\in\O_K$
agit par multiplication par $\sigma(\alpha)$. Alors (quitte \`a agrandir encore $L$)
$H^\sigma({\cal X}_\Phi)$ est de rang~$1$
sur $\O_L$;
on en fixe une base $\omega_{\sigma}$, et donc $\omega_{\sigma}$ est unique
\`a multiplication pr\`es par un \'el\'ement de $U_L$.

Si $g\in{\rm Gal}(L/\Q)$, on note $X_\Phi^g$, $\omega_{\sigma}^g$, etc. les objets
d\'eduits par extension des scalaires.  Alors $X_\Phi^g$ est isog\`ene \`a
$X_{g\Phi}$ et $\omega_\sigma^g$ est une base de $H^{g\sigma}({\cal X}_{g\Phi})$.
Si $u_g\in H^1(X_\Phi^g(\C),\Z)$ notons 
$$\langle\omega_{\sigma}^g,u_g\rangle_v:= \int_{u_g}\omega_{\sigma}^g\in \B_v$$
Les formules de Riemann impliquent que 
$$\langle\omega_\sigma^g,u_g\rangle_v
\langle\omega_{c\sigma}^g,u_g\rangle_v=\alpha_g t_v$$
 o\`u $\alpha_g\in\Qbar$, et o\`u
$t_\infty=2i\pi$ et $t_p$ est le $2i\pi$ $p$-adique de Fontaine.

Les p\'eriodes $\langle\omega_{\sigma}^g,u_g\rangle_v$ d\'ependent de tous les choix
entrant dans leur d\'efinition; nous allons les normaliser, gr\^ace au th\'eor\`eme
de Deligne~\cite{D} sur les cycles de Hodge absolus
et son avatar $p$-adique~\cite{G,B,Og,W}, pour qu'elles ne d\'ependent
que de $\Phi$ et $\sigma$ (\`a une petite ind\'etermination pr\`es).
Commen\c{c}ons par remarquer que
 la conjugaison complexe $c$ induit un isomorphisme
$H^1(X_\Phi^g(\C),\Z)\cong H^1(X_\Phi^{cg}(\C),\Z)$.

\begin{lemm}\label{Kr3}
{\rm (i)} la quantit\'e
${\rm Per}(X^g_\Phi,\omega_{\sigma}^g,\omega_{c\sigma}^{g})=
({\rm Per}_v(X^g_\Phi,\omega_{\sigma}^g,\omega_{c\sigma}^{g}))_v\in\B^\dual$, o\`u
$${\rm Per}_v(X^g_\Phi,\omega_{\sigma}^g,\omega_{c\sigma}^{g}):=
\frac{\langle\omega_{\sigma}^g,u_g\rangle_v \langle\omega_{\sigma}^{cg},c\cdot u_g\rangle_v}
{\langle\omega_{c\sigma}^g,u_g\rangle_v 
\langle\omega_{c\sigma}^{cg},c\cdot u_g\rangle_v}$$
ne d\'epend pas du choix de $u_g$, 
et ${\rm Per}_\infty(X^g_\Phi,\omega_{\sigma}^g,\omega_{c\sigma}^{g})$ est r\'eelle~$>0$.


{\rm (ii)} $v_p({\rm Per}_p(X^g_\Phi,\omega_{\sigma}^g,\omega_{c\sigma}^{g}))=
2(Z_p(b_{K,g\sigma,g\Phi},0)-\mu_{{\rm Art},p}(b_{K,g\sigma,g\Phi}))$
\end{lemm}
\begin{proof}
Que ${\rm Per}_\infty(X^g_\Phi,\omega_{\sigma}^g,\omega_{c\sigma}^{g})$ soit r\'eelle~$>0$
r\'esulte de ce que $\langle\omega_{\sigma}^g,u_g\rangle_\infty$
et $ \langle\omega_{\sigma}^{cg},c\cdot u_g\rangle_\infty$ sont des nombres complexes
conjugu\'es (et de m\^eme en rempla\c{c}ant $\sigma$ par $c\sigma$).
Que ${\rm Per}(X^g_\Phi,\omega_{\sigma}^g,\omega_{c\sigma}^{g})$
ne d\'epende pas de $u_\sigma$ est le contenu de \cite[lemme II.2.1]{fdp}: cela r\'esulte
de ce que $H^1(X_\Phi^g(\C),\Z)$
est un $K$-module de rang~$1$, et que le r\'esultat est vecteur propre pour l'action de $K$
pour le caract\`ere $g\sigma+cg\sigma-gc\sigma-cgc\sigma$ qui est trivial puisque $c$ commute
\`a tout.


Pour prouver le (ii), on utilise le (i) de \cite[th.\,II.1.1]{fdp} selon lequel,
$$v_p(\langle\omega^g_{\Phi,\sigma},u\rangle_p)=v_p(\omega^g_{g\Phi,g\sigma})+
Z_p(a_{K,g\sigma,g\Phi},0)-\mu_{{\rm Art},p}(a_{K,g\sigma,g\Phi})$$
 si
$u$ est un g\'en\'erateur du $\Z_p\otimes\O_K$-module $T_p(X_\Phi^g)$ (et idem pour les autres termes).
On a impos\'e $v_p(\omega_{\sigma}^g)=0$ en prenant une base de $H^\sigma({\cal X}_\Phi)$,
et comme le r\'esultat ne d\'epend pas du choix de $u_g$, on peut supposer
que $u_g$ est un g\'en\'erateur du $\Z_p\otimes\O_K$-module $T_p(X_\Phi^g)$ (et idem pour
les autres termes).  Le r\'esultat est alors une cons\'equence de la lin\'earit\'e
de $Z_p(\cdot,0)$ et $\mu_{{\rm Art},p}$, et des relations
$b_{K,g\sigma,g\Phi}=a_{K,g\sigma,g\Phi}-a_{K,gc\sigma,g\Phi}$
et $b_{K,cg\sigma,cg\Phi}=b_{K,g\sigma,g\Phi}$.
\end{proof}

\begin{theo}\label{per1}
{\rm (i)}
Il existe une unique application multiplicative 
$$\Omega:\CMm\to (\B\otimes\Qbar)^\dual_+/\iota({\Qbar}^\dual)_0$$
v\'erifiant,
pour tous $K$ corps CM, $\Phi\subset H_K$ type CM, $\sigma\in H_K$ et $g\in G_\Q${\rm:}
$$\Omega(b_{K,\sigma,\Phi})(g)={\rm Per}(X^g_\Phi,\omega_{\sigma}^g,\omega_{c\sigma}^g)$$

{\rm (ii)} L'application multiplicative
$$\Omega^{\rm CM}={\rm N}^{\rm CM}\circ \Omega:
\CMm\to \big((\B\otimes\Q^{\rm CM})^\dual/U^{\rm CM}\big)\otimes_\Z\Q$$
est $G_\Q$ \'equivariante {\rm(i.e.~$h*\Omega^{\rm CM}(b)=\Omega^{\rm CM}(h\cdot b)$,
avec $(h\cdot b)(g):=b(h^{-1}gh)$)}, et
co\"{\i}ncide avec $\Omega$ dans 
$\big((\B\otimes\Qbar)^\dual\otimes_\Z\Q\big)/\iota({\Qbar}^\dual)_0$
\end{theo}
\begin{proof}
Commen\c{c}ons par remarquer que le membre de droite est un \'el\'ement bien d\'efini
de $(\B\otimes\Qbar)^\dual_+/\iota({\Qbar}^\dual)_0$: 
on ne peut changer $\omega_{\sigma}$ et $\omega_{c\sigma}$ que par multiplication
par des \'el\'ements $\lambda,\lambda_c$ de $U_L$, ce qui multiplie
$\Omega(b_{K,\sigma,\Phi})(g)$ par $(\lambda\lambda_c^{-1})^g(\lambda\lambda_c^{-1})^{cg}$.  
Si on remplace $L$ par une extension, le r\'esultat ne change pas.
Enfin, si on remplace $X_\Phi$ par une vari\'et\'e isog\`ene $X'_\Phi$, l'isog\'enie
$\alpha:X_\Phi\to X'_\Phi$ fournit $\lambda_\nu\in L^\dual$ pour $\nu=\sigma,c\sigma$, tels que
$\alpha^\dual\omega'_{\Phi,\nu}=\lambda_\nu\omega_{\nu}$. 
Si $\lambda=\lambda_\sigma\lambda_{c\sigma}^{-1}$, alors
${\rm Per}((X'_\Phi)^g,(\omega'_{\sigma})^g,(\omega'_{c\sigma})^g)=\lambda^g\lambda^{cg}
{\rm Per}(X^g_\Phi,\omega_{\sigma}^g,\omega_{c\sigma}^g)$ et on a $v_p(\lambda^g\lambda^{cg})=0$
pour tous $g$ et $p$, gr\^ace au (ii) du lemme~\ref{Kr3}.

Maintenant,
l'unicit\'e est imm\'ediate car les $b_{K,\sigma,\Phi}$ engendrent le $\Z$-module $\CMm$.
Passons \`a la preuve de l'existence.
Il s'agit de v\'erifier que les relations satisfaites par les $b_{K,\sigma,\Phi}$
le sont encore par les $\Omega(b_{K,\sigma,\Phi})$.
Ces relations sont engendr\'ees par les relations suivantes:

$\bullet$ $b_{K,\sigma,\Phi}=-b_{K,c\sigma,\Phi}$.

$\bullet$ Si $h:K'\to K$ est un isomorphisme, alors $b_{K',\sigma h,\Phi h}=b_{K,\sigma,\Phi}$.

$\bullet$ Si $K\subset K'$, on a $b_{K',\sigma',\Phi'}=b_{K,\sigma,\Phi}$ si
$\sigma=\sigma'_{|K}$ et $\Phi'=\{\tau\in H_{K'},\ \tau_{|K}\in\Phi\}$.

$\bullet$ Si $\sum_i\Phi_i=\sum_j\Phi'_j$, alors 
$\sum_i b_{K,\sigma,\Phi_i}=\sum_j b_{K,\sigma,\Phi'_j}$.

La formule $\Omega(b_{K,\sigma,\Phi})=\Omega(b_{K,c\sigma,\Phi})^{-1}$ est imm\'ediate
sur la d\'efinition, et l'identit\'e $\Omega(b_{K',\sigma h,\Phi h})=\Omega(b_{K,\sigma,\Phi})$
 est la traduction
de ce que $X_{K,\Phi}$ est aussi \`a multiplication complexe par $K'$, de type CM $\Phi h$;
les deux premi\`eres relations sont donc pr\'eserv\'ees par $\Omega$.

La troisi\`eme
identit\'e \`a v\'erifier r\'esulte de ce que
$X_{K',\Phi'}$ est un produit de $[K':K]$ vari\'et\'es
ab\'eliennes isog\`enes \`a $X_{K,\Phi}$ (\'ecrire un id\'eal de $\O_{K'}$,
vu
comme $\O_K$-module, sous la forme d'une somme directe d'id\'eaux de $\O_K$). Cela
fournit une \'egalit\'e \`a multiplication pr\`es
par $\lambda^g\lambda^{cg}$, avec $\lambda\in\Qbar^\dual$,
et on montre que $\lambda\in \iota({\Qbar}^\dual)_0$ 
en constatant que $v_p(\lambda^g\lambda^{cg})=0$,
pour tous $p$ et $g$, gr\^ace au (ii) du lemme~\ref{Kr3}.
 
Il reste la derni\`ere identit\'e \`a v\'erifier (c'est aussi la plus d\'elicate).
Si $\Psi$ est un type CM de $K$, notons $h^1({\Psi})$ le 
le motif de Hodge absolu associ\'e \`a $H^1(X_\Psi)$.
Alors $h^1(\Psi)$ est de rang~$1$ sur $K$.
Si $\sum_i\Phi_1=\sum_j\Phi'_j$, et si les $X_\Psi$ pour $\Psi=\Phi_i,\Phi'_j$ sont d\'efinies sur $L$,
le th\'eor\`eme de Deligne fournit un isomorphisme 
$\otimes_i h^1(\Phi_i)\cong \otimes_j h^1(\Phi'_j)$ de motifs \`a coefficients dans $K$, d\'efinis
sur $L$.
Il existe alors $\lambda_\nu\in L^\dual$ tel que $(\otimes_i\omega_{i,\sigma})=\lambda_\nu
(\otimes_j\omega_{j,\sigma})$, si $\nu=\sigma,c\sigma$ (ce $\lambda_\nu$ d\'epend du choix
des $\omega_{\nu}$).

Si $\lambda=\lambda_{\rm id}\lambda_c^{-1}$,
alors $\prod_i {\rm Per}(X^g_{\Phi_i},\omega_{i,\sigma}^g,\omega_{i,c\sigma}^g)=
\lambda^g\lambda^{cg}\prod_j {\rm Per}(X^g_{\Phi'_j},\omega_{j,\sigma}^g,\omega_{j,c\sigma}^g)$,
pour tout $g\in G_\Q$.
Pour conclure, il reste \`a v\'erifier que $\iota(\lambda)\in\iota(\Qbar^\dual)_0$ ou,
autrement dit, que $v_p(\lambda^g\lambda^{cg})=0$ pour tous $p$ et $g$;
cela r\'esulte du (ii) du lemme~\ref{Kr3} et de la lin\'earit\'e
de $Z_p(\cdot,0)$ et $\mu_{{\rm Art},p}$.

Ceci prouve le (i); prouvons le (ii).
Si $K'$ est le plus
grand sous-corps CM de $L$, alors $K\subset K'$ et
$$(\Omega^{\rm CM}(b_{K,\sigma,\Phi}))(g)=
\Big(\prod_{h\in G_{K'}/G_L}{\rm Per}((X_\Phi^h)^g,(\omega_\sigma^h)^g,(\omega_{c\sigma}^h)^g)\Big)
^{1/[L:K']}$$
Or les $X_\Phi^h$ sont toutes isog\`enes \`a $X_\Phi$ (car $h \Phi=\Phi$),
et donc 
$${\rm Per}((X_\Phi^h)^g,(\omega_\sigma^h)^g,(\omega_{c\sigma}^h)^g)=
(\Omega(b_{K,\sigma,\Phi}))(g)\ {\rm mod}\  \iota(\Qbar^\dual)_0$$
 pour tout $h\in G_{K'}/G_L$.
Le second \'enonc\'e du (ii) s'en d\'eduit.

Il reste \`a v\'erifier la $G_\Q$-\'equivariance, et il suffit de la v\'erifier pour
les $b_{K,\sigma,\Phi}$.  Celle-ci r\'esulte des identit\'es
$b_{K,h\sigma,h\Phi}=h\cdot b_{K,\sigma,\Phi}$ et
\begin{align*}
(\Omega^{\rm CM}(b_{K,\sigma,\Phi}))(gh)={\rm Per}(X_\Phi^{gh},\omega_\sigma^{gh},\omega_{c\sigma}^{gh})
={\rm Per}(X_{h\Phi}^{g},\omega_{h\sigma}^{g},\omega_{ch\sigma}^{g})=
(\Omega^{\rm CM}(b_{K,h\sigma,h\Phi}))(g)
\end{align*}
Ceci termine la preuve du th\'eor\`eme.
\end{proof}

\section{Un r\'egulateur augment\'e}
Notons $(\C\otimes\Q^{\rm CM})^\dual_+$ l'ensemble des $\phi\in{\rm LC}(G_\Q^{\rm CM},\C^\dual)$
v\'erifiant $\phi(cg)=\phi(g)\in\R_+^\dual$, pour tout $g$.
Le (i) du lemme~\ref{Kr3} implique que
la composante $\Omega_\infty^{\rm CM}$ de $\Omega^{\rm CM}$ (i.e.~sa projection sur
$((\C\otimes\Q^{\rm CM})^\dual/U^{\rm CM})\otimes_\Z\Q$)
est \`a valeurs dans $((\C\otimes\Q^{\rm CM})^\dual_+/U^{\rm CM})\otimes_\Z\Q$.

\begin{lemm}\label{per2}
Soit $K$ un corps CM.

{\rm (i)}
$\Omega^{\rm CM}_\infty(b_{K,{\rm id},{\rm id}})$ admet un rel\`evement 
$\Omega_\infty(K)\in(\C\otimes\Q^{\rm CM})^\dual_+$ invariant par $G_K^{\rm CM}$.

{\rm (ii)} $\Omega_\infty(K)$ est unique \`a multiplication pr\`es par un \'el\'ement
de $U_K^+\otimes\Q$.
\end{lemm}
\begin{proof}
Remarquons que $h\cdot b_{K,{\rm id},{\rm id}}=b_{K,{\rm id},{\rm id}}$ si $h\in G_K^{\rm CM}$.
Il s'ensuit que tout rel\`evement de $\Omega^{\rm CM}_\infty(b_{K,{\rm id},{\rm id}})$ est
invariant par $G_K^{\rm CM}$ modulo $U^{\rm CM}\otimes\Q$.
Partons donc d'un tel rel\`evement $\tilde\Omega_\infty(K)$; il est invariant
par $G_L^{\rm CM}$ o\`u $L\subset K$ est un corps CM.
Mais alors 
$$\Big(\prod_{h\in G_K^{\rm CM}/G_L^{\rm CM}}h*\tilde\Omega_\infty(K)\Big)^{1/[L:K]}$$
 est
invariant par $G_K^{\rm CM}$ et est un rel\`evement de $\Omega^{\rm CM}_\infty(b_{K,{\rm id},{\rm id}})$
puisque chacun des termes du produit en est un.  Ceci prouve le (i).

Le (ii) est une cons\'equence du fait qu'un \'el\'ement de $U^{\rm CM}\otimes_\Z\Q$
fixe par $G_K^{\rm CM}$ appartient \`a $U_K^+\otimes\Q$ (c'est imm\'ediat si on ne tensorise
pas par $\Q$, et on peut toujours \'elever notre \'el\'ement \`a une puissance enti\`ere
pour passer de $U^{\rm CM}\otimes_\Z\Q$ \`a $U^{\rm CM}$). 
\end{proof}

Comme $\Omega_\infty(K)$ est invariant par $G_K^{\rm CM}$, on peut le voir
comme un \'el\'ement de $(\R_+^\dual)^{H_K}$ ou encore comme un \'el\'ement
de $(\C\otimes K)^\dual_+$. Vu dans $(\R_+^\dual)^{H_K}$, on a
$$\Omega_\infty(K)\sim(\Omega^{\rm CM}_\infty(b_{K,\sigma,\sigma}))_{\sigma\in H_K}$$

\begin{prop}\label{Kr6}
Soit $v_1,\dots,v_{d-1}$ formant une famille libre dans $U_K$, alors
$$R_K{\rm ht}(\chi)=\frac{1}{[U_K:\langle v_1,\dots,v_{d-1}\rangle]}
|\det({\rm Log}|v_1|^2,\dots{\rm Log}|v_{d-1}|^2,{\rm Log}(\Omega_\infty(K)))|$$
\end{prop}
\begin{proof}
Remarquons que le membre de droite ne d\'epend pas du choix de $\Omega_\infty(K)$ puisque
cette quantit\'e est unique \`a multiplication pr\`es par un \'el\'ement de $U_K^+\otimes\Q$,
dont le logarithme est une combinaison lin\'eaire des autres colonnes.

Si on fait la somme des lignes, tous les termes sont $0$ sauf le terme de la derni\`ere
colonne qui vaut $\log({\rm N}_{K/\Q}\Omega_\infty(K))$.
Vu la d\'efinition de $R_K$, il suffit de v\'erifier que
$\log({\rm N}_{K/\Q}\Omega_\infty(K))={\rm ht}(\chi)$.

La fonction $b_{K,\rm id,id}$ est nulle en dehors de $G_F$ et co\"{\i}ncide avec $\chi$
sur $G_F$.  Si $h\in G_\Q$, on a $b_{K,h\cdot{\rm id},h\cdot{\rm id}}=b_{K,\rm id,id}(h^{-1}gh)$.
On en d\'eduit que $\sum_{\sigma\in H_F}b_{K,\sigma,\sigma}={\rm Ind}_{G_F}^{G_\Q}\chi$.
En particulier, $\sum_{\sigma\in H_F}b_{K,\sigma,\sigma}$ est une fonction centrale,
et donc l'identit\'e que l'on cherche \`a d\'emontrer est un cas particulier du
dernier point de la prop.\,\ref{CM5}.
\end{proof}

\section{P\'eriodes absolues}
Comme dans le th.\,\ref{per1},
on fait agir $G_\Q$ sur $\CMm$ par conjugaison int\'erieure sur la variable:
 $(h\cdot a)(g)=a(h^{-1}gh)$.
Si $a\in\CMm$, on note $\Gamma(a)\subset G_\Q$ le stabilisateur de $a$, et
on pose $F(a)=\Qbar^{\Gamma(a)}$ (alors $F(a)$ est totalement r\'eel car
l'action de $G_\Q$ sur $\CMm$ se factorise par $G_\Q^{\rm CM}$ 
dans lequel $c$ est central). 
\begin{rema}\label{CM4}
(i) $a\in \CMm^0$ si et seulement si $F(a)=\Q$.

(ii) Si $h\in G_\Q/\Gamma(a)=H_{F(a)}$, alors $\Gamma(h\cdot a)=h\Gamma(a)h^{-1}$
et $F(h\cdot a)=F(a)^h$.
\end{rema}

Si $a\in\CMm$, choisissons un rel\`evement $\tilde\Omega_\infty(a)$ de
$\Omega_\infty^{\rm CM}(a)$ dans $(\C\otimes\Q^{\rm CM})^\dual_+$ comme ci-dessus.
Il existe alors un corps CM $K$ tel que $\tilde\Omega_\infty(a)$ soit
fixe par $G_K^{\rm CM}$.  On pose, comme ci-dessus,
$$\Omega_\infty^{\rm abs}(a)=
\Big(\prod_{h\in G_{F(a)}^{\rm CM}/G_K^{\rm CM}}h*\tilde\Omega_\infty(a)\Big)^{1/[K:F(a)]}$$

\begin{prop}\label{CM5}
{\rm (i)} 
$\Omega_\infty^{\rm abs}(a)$ est fixe par $G_{F(a)}^{\rm CM}$ et ne d\'epend pas,
\`a multiplication pr\`es par un \'el\'ement de $U_{F(a)}^+\otimes_\Z\Q$,
du choix de $\tilde\Omega_\infty(a)$.

{\rm (ii)}
$\Omega_\infty^{\rm abs}:\CMm\to\R_+^\dual$ v\'erifie les propri\'et\'es suivantes:

\quad $\bullet$ si $a\in\CMm$, alors
$\Omega_\infty^{\rm abs}(a)$ 
a pour image $\Omega^{\rm CM}_\infty(a)$ mod~$U^{\rm CM}\otimes_\Z\Q$.

\quad $\bullet$ $\Omega_\infty^{\rm abs}(a+b)
=({\rm N}_{F(a,b)/F(a+b)}(\Omega_\infty^{\rm abs}(a)\Omega_\infty^{\rm abs}(b)))^{1/[F(a,b):F(a+b)]}$
mod~$U_{F(a+b)}^+\otimes\Q$.

\quad $\bullet$ Si $a\in\CMm^0$, alors $\Omega_\infty^{\rm abs}(a)=\exp {\rm ht}(a)$.
\end{prop}
\begin{proof}
Le (i) est imm\'ediat ainsi que
les deux premi\`eres propri\'et\'es du (ii).  La derni\`ere
est juste un exercice de traduction entre la d\'efinition de ${\rm ht}$ (cf.~\cite[th.\,0.3]{fdp})
et celle
de $\Omega_\infty^{\rm abs}$: les deux font intervenir une moyenne des
${\rm Per}_\infty(X_\Phi^g,\omega_\sigma^g,\omega_{c\sigma}^g)$ 
(not\'e $|\langle\omega_\sigma^g,\omega_{c\sigma}^g,u_g\rangle_\infty|_\infty$ dans~\cite{fdp}),
et les facteurs $v_p(\omega_\sigma^g)-v_p(\omega_{c\sigma}^g)$ qui apparaissent
dans la d\'efinition de ${\rm ht}$ de~\cite{fdp} sont $0$ dans cet article
gr\^ace \`a notre normalisation des $\omega_\tau$.
\end{proof}

\begin{rema}\label{CM6}
Yoshida~\cite{yoshida1} a \'enonc\'e une conjecture
donnant une formule pour des variantes de $\Omega^{\rm CM}_\infty(b_{K,\sigma,\tau})$
en termes des fonctions Gamma multiples de Shintani mentionn\'ees dans~\cite[\S\,0.7]{fdp}.
Notons que les variantes qu'il consid\`ere sont aussi d\'efinies ``\`a une unit\'e pr\`es''.
\end{rema}

\appendix
\section{}
La conj.\,0.4 de~\cite{fdp}, combin\'ee avec le (ii) du th.\,0.3 de~\cite{fdp}, 
fournit une formule conjecturale
pour la hauteur de Faltings d'une vari\'et\'e ab\'elienne \`a multiplication
complexe par l'anneau des entiers de
son corps de multiplication complexe. 
On peut se demander
``quelle portion'' de la conj.\,0.4 d\'ecoule de la conjecture sur les hauteurs de Faltings;
la prop.\,\ref{per3} ci-dessous
 montre que c'est ``la moiti\'e'' (par contre, la conjecture ``en moyenne'',
d\'emontr\'ee dans~\cite{AGHM,YZ},
ne fournit qu'une toute petite partie de cette conj.\,0.4).

On note $a\mapsto a^0$ la projection de $\Q\otimes\CM$ sur $\Q\otimes\CM^0$; 
elle envoie $\Q\otimes\CMm$ dans $\Q\otimes\CMm^0$.
Si $K$ est un corps CM et $\Phi\subset K$ est un type CM,
on pose\footnote{\label{per20} On a $B_{K,\Phi}(g)=2A_{K,\Phi}(g)-|\Phi|$, et la projection
$A_{K,\Phi}^0$ de $A_{K,\Phi}$
est la fonction reli\'ee \`a la hauteur de Faltings de $X_\Phi$ dans~\cite{fdp}.
En particulier, la conjecture pour les hauteurs de Faltings \'equivaut \`a
${\rm ht}(B_{K,\Phi}^0)=Z(B_{K,\Phi}^0,0)$ pour tous $K$ et $\Phi$.}
 $B_{K,\Phi}=\sum_{\sigma,\tau\in\Phi}b_{K,\sigma,\tau}$
o\`u $b_{K,\sigma,\tau}$ est d\'efinie par la formule~(\ref{truc5}).

\begin{prop}\label{per3}
Le sous-espace de $\Q\otimes\CMm^0$ engendr\'e par les $B_{K,\Phi}^0$ est l'espace
des fonctions paires de $\Q\otimes\CMm^0$
{\rm (i.e. $\phi(g^{-1})=\phi(g)$, pour tout $g\in G_\Q$)}.
\end{prop}
\begin{proof}
On a $b_{K,\sigma,\tau}(g^{-1})=b_{K,\tau,\sigma}(g)$; on en d\'eduit que
$B_{K,\Phi}$
est paire.
Comme sym\'etriser une fonction paire donne une fonction paire, $B_{K,\Phi}^0$ est aussi paire.
Pour la m\^eme raison, il suffit de v\'erifier que
le sous-espace de $\Q\otimes\CMm$ engendr\'e par les $B_{K,\Phi}$ est l'espace
des fonctions paires de $\Q\otimes\CMm$.

Maintenant, les $b_{K,\sigma,\tau}$ engendrant $\Q\otimes\CMm$, 
les $b_{K,\sigma,\tau}+b_{K,\tau,\sigma}$ forment une famille g\'en\'eratrice
de l'espace des fonctions paires de $\CMm$.

Si on \'ecrit $\Phi=\Psi\sqcup\{\alpha,\beta\}$, on fabrique des types CM
$$\Phi_{00}=\Phi,\quad \Phi_{01}=\Psi\sqcup\{\alpha,c\beta\},\quad
\Phi_{10}=\Psi\sqcup\{c\alpha,\beta\},\quad \Phi_{11}=\Psi\sqcup\{c\alpha,c\beta\}$$
Alors
$$B_{K,\Phi_{00}}-B_{K,\Phi_{10}}-B_{K,\Phi_{01}}+B_{K,\Phi_{11}}
=4(b_{K,\alpha,\beta}+b_{K,\beta,\alpha})$$
On en d\'eduit que l'espace engendr\'e par les $B_{K,\Phi}$ contient
une famille g\'en\'eratrice de l'espace des fonctions paires de $\Q\otimes\CMm$, ce qui permet
de conclure.
\end{proof}
\begin{rema}\label{per4}
Il est facile d'\'ecrire une fonction paire comme combinaison lin\'eaire
de $1$ et des $b_{K,\sigma,\tau}+b_{K,\tau,\sigma}$; 
la preuve de la proposition fournit une formule permettant de transformer
une telle expression en une
combinaison lin\'eaire des $B_{K,\Phi}$ (ou des $A_{K,\Phi}$, cf.~note~\ref{per20}). 
 Projeter sur $\CM^0$ fournit
une expression d'une fonction centrale paire comme combinaison
lin\'eaire des $A_{K,\Phi}^0$.
\end{rema}

\end{document}